\documentclass[11pt]{article}
\usepackage[T1]{fontenc}
\usepackage{lmodern}
\usepackage{amsmath,amssymb,amsthm,mathtools}
\usepackage[margin=1in]{geometry}
\usepackage{microtype}
\usepackage[colorlinks=true,citecolor=blue,linkcolor=blue,urlcolor=blue]{hyperref}

\newtheorem{theorem}{Theorem}[section]
\newtheorem{lemma}[theorem]{Lemma}
\newtheorem{corollary}[theorem]{Corollary}
\theoremstyle{remark}
\newtheorem{remark}[theorem]{Remark}

\newcommand{\Kc}{\mathbf K}
\newcommand{\Ec}{\mathbf E}
\newcommand{\M}{\operatorname{AG}}
\newcommand{\dd}{\,\mathrm d}

\title{Global convergence and monotonicity of Newton iteration
for the inverse Gr\"otzsch modulus}
\author{Deguang Zhong\textsuperscript{1,*}\\
\quad
\small\textsuperscript{1}Institute of Applied Mathematics, Shenzhen Polytechnic University,
Shenzhen 518055,  China
}
\date{}

\begin{document}
\maketitle

\begin{abstract}
Let
\[
 \mu(r)=\frac{\pi}{2}\frac{\Kc(r')}{\Kc(r)},
 \qquad r'=\sqrt{1-r^2},\qquad 0<r<1,
\]
be the Gr\"otzsch modulus function.  Problems 3.38(a)--(b) in a survey of
Vuorinen ask whether Newton's iteration for $\mu^{-1}(y)$, initialized by
$x_0=1/\cosh y$, converges for every $y>\pi/2$, and whether it is strictly
increasing when $y>\pi$.  We prove both assertions.  The main point is that
$\mu$ has exactly one inflection point $a\in(1/2,1/\sqrt2)$.  If the zero lies
in the convex region, the Newton iterates increase to it from the left.  If
the zero lies in the concave region, the orbit crosses the inflection point,
overshoots the zero at most once, and then decreases to the zero.  For
$y>\pi$ we obtain the stronger estimate $0<x_n<x_{n+1}<\mu^{-1}(y)<3-2\sqrt2<1.$
\end{abstract}

\section{Introduction}

The function $\mu$ occurs throughout planar quasiconformal analysis.  It is the module of the Gr\"{o}otzsch extremal domain $\mathbb{B}\setminus[0,r]$ and enters, for example, the
Hersch--Pfluger distortion function \cite{AVV1988,  HP52, LV73}
\[
 \varphi_K(r)=\mu^{-1}\!\left(\frac{\mu(r)}{K}\right).
\]
Consequently, reliable inversion of $\mu$ is a basic step in the numerical
evaluation of sharp quasiconformal distortion bounds; see
\cite{AVVbook, AQVV2000, VuorinenSurvey}.  Kargar, Rainio, and Vuorinen \cite[Sections~1 and~3]{KRV2024}
recall that Newton iteration was used in \cite{AVVbook} for the
numerical evaluation of \(\mu^{-1}\) and \(\varphi_K\), and compare
their new Landen-based approximations with those earlier
implementations.

In \cite[Problems 3.38(a)--(b)]{VuorinenSurvey}, Vuorinen considered Newton's
method
\begin{equation}\label{eq:newton}
 x_0=\frac1{\cosh y},
 \qquad
 x_{n+1}=x_n-\frac{\mu(x_n)-y}{\mu'(x_n)},
 \qquad y>\frac\pi2,
\end{equation}
and asked:
\begin{enumerate}
 \item[(a)] Does \eqref{eq:newton} converge to $\mu^{-1}(y)$ for every
 $y>\pi/2$?
 \item[(b)] Is $x_n<x_{n+1}<1$ for all $n$ whenever $y>\pi$?
\end{enumerate}
To the best of our knowledge, no proof of the global convergence
asserted in \cite[Problem 3.38(a)]{VuorinenSurvey}, nor of the monotonicity assertion in
\cite[Problem 3.38(b)]{VuorinenSurvey}, has previously appeared in the literature. We answer both questions affirmatively.  Our proof is global: it identifies
the full qualitative geometry of $\mu$ needed by Newton's method, rather than
invoking a local convergence theorem whose hypotheses would have to be
verified near an unknown zero.

There is a sign issue worth recording at the outset.  If
$\M(1,r')$ denotes the arithmetic--geometric mean, then
\begin{equation}\label{eq:mu-prime-intro}
 \mu'(r)=-\frac{\M(1,r')^2}{r(1-r^2)}.
\end{equation}
Here, the definition of the arithmetic geometric mean of $a$ and $b$ is defined by follows: For $0<b<a$ define $a_0=a$, $b_0=b$, $a_{n+1}=(a_n+b_n)/2$ and $b_{n+1}=\sqrt{a_n b_n}$. Then, the common limit of these sequences
\begin{equation*}
  AG(a,b)=\lim_{n\rightarrow\infty}a_n=\lim_{n\rightarrow\infty}b_n,
\end{equation*}
is called the {\it arithmetic geometric mean} of $a$ and $b$; see \cite{BorweinBorwein}.  It follows that the AGM form of \eqref{eq:newton} is
\begin{equation}\label{eq:correct-agm}
 x_{n+1}=x_n+
 \frac{\bigl(\mu(x_n)-y\bigr)(x_n-x_n^3)}{\M(1,x_n')^2}.
\end{equation}
The displayed AGM formula in \cite[Section 3.37]{VuorinenSurvey} has a minus
sign in place of the plus sign in \eqref{eq:correct-agm}, although its first
expression is the canonical Newton step \eqref{eq:newton}.  We therefore
interpret the question, as its wording requires, by the first expression.
The literal minus-sign iteration cannot converge to the desired zero: as
shown below, $x_0<\mu^{-1}(y)$ and $\mu(x_0)>y$, so that its first step moves
strictly to the left, away from the zero.

\section{Some facts about elliptic integrals}
In this section, we recall some facts about elliptic integrals. The material presented in this section is adapted from \cite{AVVbook} and \cite{HKV20}.
For $0<r<1$, let
\begin{align*}
 \Kc(r)&=\int_0^{\pi/2}\frac{\dd t}
 {\sqrt{1-r^2\sin^2t}},\\
 \Ec(r)&=\int_0^{\pi/2}\sqrt{1-r^2\sin^2t}\,\dd t,
 \qquad r'=\sqrt{1-r^2},
\end{align*}
be the Legendre's complete elliptic integrals of the first and second kinds; see \cite[Formula (3.1)-(3.2)]{AVVbook}.  We use the
standard differentiation formulas (see \cite[Formula (3.6)-(3.7)]{AVVbook})
\begin{equation}\label{eq:KE-derivatives}
 \Kc'(r)=\frac{\Ec(r)-r'^2\Kc(r)}{rr'^2},
 \qquad
 \Ec'(r)=\frac{\Ec(r)-\Kc(r)}r,
\end{equation}
and Legendre's relation \cite[Formula (3.3)]{AVVbook}
\begin{equation}\label{eq:legendredcsdcs}
 \Kc(r)\Ec(r')+\Ec(r)\Kc(r')
 -\Kc(r)\Kc(r')=\frac{\pi}{2}.
\end{equation}
 From \cite[Formula (3.10)]{AVVbook}, we see that by differentiating
$\mu(r)=(\pi/2)\Kc(r')/\Kc(r)$ gives 
\begin{equation}\label{eq:mu-prime}
 \mu'(r)=-\frac{\pi^2}{4rr'^2\Kc(r)^2}
 =-\frac{\M(1,r')^2}{rr'^2}<0,
\end{equation}
where $\Kc(r)=\pi/(2\M(1,r'))$.  Thus $\mu$ is a decreasing
homeomorphism from $(0,1)$ onto $(0,\infty)$.

We shall also use the ascending Landen identity \cite[p. 51]{AVVbook}
\begin{equation}\label{eq:landen}
 \mu\!\left(\frac{2\sqrt{s}}{1+s}\right)=\frac12\mu(s),
 \qquad 0<s<1.
\end{equation}
With equation (\ref{eq:landen}), we have the following result.
\begin{lemma}\label{lem:initial}
If $y>0$, $\rho=\mu^{-1}(y)$, and
$x_0=1/\cosh y$, then $0<x_0<\rho$.
\end{lemma}
\begin{proof}
A known result from  \cite[P. 122, (7.21)]{HKV20}, which states  that  for every $0<r<1$,
\begin{equation}\label{eq:mu-log}
 \mu(r)>\log\frac1r.
\end{equation}
Now put $s=e^{-2y}$.  Since
$$
 \frac1{\cosh y}=\frac{2e^{-y}}{1+e^{-2y}}
 =\frac{2\sqrt{s}}{1+s},
$$
the Landen identity (\ref{eq:landen}) and inequality \eqref{eq:mu-log} yield
$$
 \mu(x_0)=\frac{1}{2}\mu(s)>\frac{1}{2}\log\frac{1}{s}
 =y=\mu(\rho).
$$
Because $\mu$ is decreasing, $x_0<\rho$.
\end{proof}

\section{The unique inflection point}

The following lemma supplies the geometric input for the global Newton
argument.

\begin{lemma}\label{lem:inflection}
There exists a unique $a\in(1/2,1/\sqrt2)$ such that
\begin{equation}\label{eq:convex-concave}
 \mu''(r)>0\quad(0<r<a),
 \qquad
 \mu''(r)<0\quad(a<r<1).
\end{equation}
Moreover, if $c=1/\sqrt2$ and
\[
 N_y(r)=r-\frac{\mu(r)-y}{\mu'(r)},
\]
then
\begin{equation}\label{eq:key-barrier}
 N_{\pi/2}(a)<1.
\end{equation}
\end{lemma}

\begin{proof}
Write
\[
 h(r)=rr'^2\Kc(r)^2,
 \qquad
 u(r)=\frac{\Ec(r)}{\Kc(r)},
 \qquad
 q(r)=2u(r)-(1+r^2).
\]
Equations \eqref{eq:KE-derivatives} give
\begin{equation}\label{eq:h-log}
 \frac{h'(r)}{h(r)}=\frac{q(r)}{rr'^2}.
\end{equation}
Since $\mu'=-\pi^2/(4h)$, it follows that
\begin{equation}\label{eq:musecond-sign}
 \operatorname{sgn}\mu''(r)=\operatorname{sgn}q(r).
\end{equation}
On the other hand, differentiating $u=\Ec/\Kc$ and using
\eqref{eq:KE-derivatives}, we find
\begin{equation}\label{eq:qprime}
 q'(r)=-\frac{2\{(u(r)-r'^2)^2+2r^2r'^2\}}{rr'^2}<0.
\end{equation}
Thus $q$ is strictly decreasing, and $q(0^+)=1$.

We next locate the zero of $q.$  Since $\Ec(r)-r'^2\Kc(r)>0$, one has $u(r)>r'^2$.  At $r=1/2$ this
gives
\[
 q(1/2)>2\cdot\frac34-\frac54=\frac14>0.
\]
At $c=1/\sqrt2$, the Legendre's relation (\ref{eq:legendredcsdcs}) reduces to
\[
 2\Kc(c)\Ec(c)-\Kc(c)^2=\frac\pi2,
\]
and hence
\[
 q(c)=\frac{\pi}{2\Kc(c)^2}-\frac12<0.
\]
For completeness, the last inequality follows already from the first three
positive terms of the hypergeometric series:
\[
 \Kc(c)>
 \frac\pi2\left(1+\frac18+\frac9{256}\right)
 =\frac{297\pi}{512}>\sqrt\pi.
\]
Therefore $q$ has a unique zero $a\in(1/2,c)$, and
\eqref{eq:convex-concave} follows from \eqref{eq:musecond-sign}.

It remains to prove the inequality \eqref{eq:key-barrier}.  Because
$\mu(c)=\pi/2$ and $\mu''<0$ on $(a,c)$, the function
$|\mu'|$ is increasing there.  Consequently,
\begin{equation}\label{eq:delta-bound}
 \mu(a)-\frac\pi2
 =\int_a^c|\mu'(t)|\dd t
 \le(c-a)|\mu'(c)|.
\end{equation}
Furthermore, the equation \eqref{eq:mu-prime} gives
\[
 \frac{|\mu'(c)|}{|\mu'(a)|}=\frac{h(a)}{h(c)}<2.
\]
Indeed, the elementary bounds
\[
 \frac\pi2<\Kc(r)<\frac\pi{2r'}
\]
imply
\[
 h(a)<\frac{\pi^2a}{4},
 \qquad
 h(c)>\frac{\pi^2c}{8},
 \qquad
 \frac{h(a)}{h(c)}<\frac{2a}{c}<2.
\]
Combining this with \eqref{eq:delta-bound}, we obtain
\[
 N_{\pi/2}(a)
 =a+\frac{\mu(a)-\pi/2}{|\mu'(a)|}
 <a+2(c-a)=2c-a
 <\sqrt2-\frac12<1,
\]
where we used $a>1/2$.  This proves \eqref{eq:key-barrier}.
\end{proof}

\section{Solution of the open problems}

\begin{theorem}\label{thm:main}
Let $y>\pi/2$, let $\rho=\mu^{-1}(y)$, and define $(x_n)$ by
\eqref{eq:newton}.  Then every iterate is well defined and belongs to
$(0,1)$, and
\[
                x_n\longrightarrow\rho.
\]
More precisely, with $a$ as in Lemma~\ref{lem:inflection}, the following
alternatives hold.
\begin{enumerate}
 \item If $\rho\le a$, then
 \[
       0<x_0<x_1<\cdots<x_n<\rho
 \]
 for every $n$.
 \item If $\rho>a$, then the orbit reaches $[a,1)$ after finitely many
 steps, crosses to the right of $\rho$ at most once, and thereafter decreases
 strictly to $\rho$.
\end{enumerate}
In either case the convergence is eventually quadratic.
\end{theorem}

\begin{proof}
Put $f(r)=\mu(r)-y$ and
\[
 N(r)=r-\frac{f(r)}{f'(r)}.
\]
The root $\rho$ is simple because $f'=\mu'<0$.  A direct differentiation
gives the useful identity
\begin{equation}\label{eq:Nprime}
 N'(r)=\frac{f(r)f''(r)}{f'(r)^2}.
\end{equation}
Lemma~\ref{lem:initial} gives $0<x_0<\rho$.

Suppose first that $\rho\le a$.  The function $f$ is strictly decreasing and
strictly convex on $(0,\rho)$.  For $0<x<\rho$, one has $f(x)>0$ and hence
$N(x)>x$.  Convexity gives
\[
 0=f(\rho)>f(x)+f'(x)(\rho-x),
\]
so the zero of the tangent line at $x$ lies strictly before $\rho$; that is,
$N(x)<\rho$.  Induction yields
\[
 x_n<x_{n+1}<\rho.
\]
The sequence therefore has a limit $L\le\rho$.  Passing to the limit in
$x_{n+1}-x_n=-f(x_n)/f'(x_n)$ shows that $f(L)=0$, whence $L=\rho$.

Now assume that $\rho>a$.  Since $y>\pi/2$ and $\mu(c)=\pi/2$, where
$c=1/\sqrt2$, we have
\begin{equation}\label{eq:root-location}
             a<\rho<c.
\end{equation}
 Moreover,
$x_0=1/\cosh y<1/\cosh(\pi/2)<\frac12<a.$
Indeed,
$\cosh(\pi/2)
 >1+(\pi/2)^2/2
 =1+\frac{\pi^2}{8}>2.$
Thus the orbit starts in the convex interval $(0,a)$.  Strict concavity of $f$ on $[a,\rho]$ implies that the tangent line at $a$
meets the horizontal axis to the right of $\rho$, so
\begin{equation}\label{eq:Na-rho}
 \rho<N(a).
\end{equation}
On the other hand, $y>\pi/2$ and Lemma~\ref{lem:inflection} give
\begin{equation}\label{eq:Na-one}
 N(a)=a+\frac{\mu(a)-y}{|\mu'(a)|}
 <N_{\pi/2}(a)<1.
\end{equation}

On $(0,a)$, formula \eqref{eq:Nprime} shows that $N'>0$, because
$f>0$ and $f''>0$.  Moreover $N(x)>x$ there.  If an orbit starting below
$a$ never reached $a$, it would increase to a limit $L\le a$.  Taking limits
in the Newton recurrence would give $f(L)=0$, contradicting $\rho>a$.
Hence the orbit enters $[a,1)$ after finitely many steps; by the monotonicity
of $N$ on $(0,a)$ and \eqref{eq:Na-one}, its first point in that interval is
smaller than $N(a)$.

For $a\le x<\rho$, strict concavity gives
\begin{equation}\label{eq:left-concave}
              \rho<N(x)\le N(a)<1.
\end{equation}
For $\rho<x\le N(a)$, the same tangent-line argument, now made from the
right of the root, gives
\begin{equation}\label{eq:right-concave}
              \rho<N(x)<x.
\end{equation}
Thus, after at most one crossing of the root, the iterates lie in
$(\rho,N(a)]$ and decrease to a limit.  Passing to the limit in the
recurrence again identifies that limit with $\rho$.  All iterates remain in
$(0,1)$ by \eqref{eq:Na-one}.  Finally, eventual quadratic convergence is the
standard local Newton estimate, since $f$ is smooth near the simple zero
$\rho$.
\end{proof}

\begin{corollary}[Solution of Problem 3.38(b)]\label{cor:monotone}
If $y>\pi$, then
\begin{equation}\label{eq:strong-monotone}
 0<x_n<x_{n+1}<\mu^{-1}(y)<3-2\sqrt2<1
 \qquad(n\ge0).
\end{equation}
\end{corollary}

\begin{proof}
At $c=1/\sqrt2$ one has $\mu(c)=\pi/2$.  The descending Landen
transformation (\ref{eq:landen}) therefore gives
\[
 \mu(3-2\sqrt2)
 =\mu\!\left(\frac{1-c}{1+c}\right)
 =2\mu(c)=\pi.
\]
Since $\mu$ is decreasing, $y>\pi$ implies
\[
 \rho=\mu^{-1}(y)<3-2\sqrt2<\frac12<a.
\]
We are consequently in the convex alternative of
Theorem~\ref{thm:main}, which proves \eqref{eq:strong-monotone}.
\end{proof}

\begin{remark}[The printed minus-sign formula]
If the second display in \cite[Section 3.37]{VuorinenSurvey} is interpreted
literally, define
\[
 \widetilde N_y(x)=x-
 \frac{(\mu(x)-y)x(1-x^2)}{\M(1,x')^2}.
\]
Lemma~\ref{lem:initial} gives $\mu(x_0)>y$ and $x_0<\rho$, hence
$\widetilde N_y(x_0)<x_0<\rho$.  As long as a literal iterate remains in
$(0,\rho)$ it continues to move left, and if it leaves that interval the
formula is no longer an iteration of $\mu$ on $(0,1)$.  It therefore cannot
converge to $\rho$.  The plus sign in \eqref{eq:correct-agm} is essential.
\end{remark}

\section{Concluding remarks}

The proof reveals why the numerical iteration is stable over the whole range
$y>\pi/2$.  The threshold $\pi/2$ places the target zero to the left of the
self-complementary point $1/\sqrt2$.  Within that interval, $\mu$ has only one
change of convexity.  Newton's method is monotone on the convex side; on the
short concave side it may overshoot, but the barrier $N_{\pi/2}(a)<1$ traps the
orbit in the concave basin to the right of the zero.  This mechanism also
shows that eventual monotonicity holds for all $y>\pi/2$, although global
monotone increase is guaranteed in the range where the zero lies before the
inflection point.  The condition $y>\pi$ in Problem 3.38(b) is a convenient
explicit subrange of the latter regime.

\medskip
\noindent\textbf{Acknowledgments.}
This work was supported by Guangdong Basic and Applied Basic Research Foundation (No. 2022A1515110967 and No. 2023A1515011809). We used GPT-5.6-sol to 
explore the Problems 3.38(a)--(b) in \cite{VuorinenSurvey}.  All mathematical arguments and proofs in the final manuscript were checked by the authors.

\end{document}